\documentclass[12pt,reqno]{article}

\usepackage{amssymb}
\usepackage{amsmath}
\usepackage{amsthm}
\usepackage{amsfonts}
\usepackage{authblk}
\usepackage{hyperref}

\theoremstyle{plain}
\newtheorem{theorem}{Theorem}
\newtheorem{prop}[theorem]{Proposition}
\newtheorem{lem}[theorem]{Lemma}

\theoremstyle{definition}

\newtheorem{rem}[theorem]{Remark}

\makeatletter
\newcommand{\subjclass}[2][2010]{%
  \let\@oldtitle\@title%
  \gdef\@title{\@oldtitle\footnotetext{#1 \emph{Mathematics subject classification.} #2}}%
}
\newcommand{\keywords}[1]{%
  \let\@@oldtitle\@title%
  \gdef\@title{\@@oldtitle\footnotetext{\emph{Key words and phrases.} #1.}}%
}
\makeatother

\begin{document}

\title{A sharp estimate of the discrepancy of a certain numerical sequence}
\author[1]{Martin Lind}
\affil[1]{Department of Mathematics and Computer Science, Karlstad University, Universitetsgatan 2, 65188 Karlstad, Sweden\\ e-mail: \href{mailto:martin.lind@kau.se}{martin.lind@kau.se}}
\date{}
\maketitle

 
\begin{abstract}

We consider a certain equidistributed sequence of rational numbers constructed from the primes. In particular, we determine the sharp convergence rate for the star discrepancy of said sequence. Our arguments are based on well-known discrepancy estimates for inversive congruential pseudorandom numbers together with asymptotic formulae involving prime numbers. 
\end{abstract}

\section{Introduction}

\emph{Equidistribution} is a fundamental idea that is central in several fields of mathematics.
In number theory, the historical origin point of the concept is Diophantine approximation; by now equidistribution plays an important role in various parts of the field (see  \cite{EqNT}).
In this note, we consider the least complicated setting in which one can study equidistribution: sequences of real numbers. Our aim is simply to compute the \emph{discrepancy} of a certain equidistributed sequence of rational numbers contained in $[0,1]$ and constructed from the prime numbers.

We start by recalling some basic definitions. Consider a numerical sequence $\xi=\{\xi_n\}_{n=1}^\infty$ where each $\xi_n\in[0,1]$. For any interval $J\subseteq[0,1]$ and $N\in\mathbb{N}$, we set
$$
A_N(\xi,J)=\sharp(\{\xi_1,\xi_2,...,\xi_N\}\cap J),
$$
where $\sharp E$ denotes the cardinality of the finite set $E$ (counted with multiplicity).
The sequence $\xi$ is said to be \emph{equidistributed} in $[0,1]$ (or uniformly distributed in $[0,1]$) if
$$
\lim_{N\rightarrow\infty}\frac{A_N(\xi,J)}{N}=\text{length}(J)
$$
for every interval $J\subseteq[0,1]$.
A quantitative measure of equidistribution is given by the \emph{star discrepancy}
\begin{equation}
    \label{discrepancyDefn}
    D_N^*(\xi)=\sup_{0<r\le1}\left|\frac{A_N(\xi,r)}{N}-r\right|,
\end{equation}
where $A_N(\xi,r)=A_N(\xi,[0,r])$. One can show that (see \cite[Chapter 2.1]{KN})
\begin{equation}
    \label{equidistrChar}
    \xi\text{ is equidistributed in }[0,1]\Leftrightarrow\lim_{N\rightarrow\infty}D_N^*(\xi)=0.
\end{equation}
Moreover, the convergence rate of $D_N^*(\xi)$ provides a measure of the extent to which $\xi$ is equidistributed.

There is much literature on discrepancy theory, e.g., the excellent texts \cite{DP} and \cite{KN}. We also mention \cite{NW} which contains many applications of number theory. Particularly relevant for this note are discrepancy estimates for pseudorandom numbers \cite[Chapter 5]{NW} (see Lemma \ref{inverseLemma} below).

Say that one wishes to construct a simple example of a sequence that is equidistributed in $[0,1]$. A natural idea is to concatenate 'blocks' (i.e., finite sequences) of equally spaced rational numbers in $[0,1]$ having the same denominator, i.e.,
$$
\omega=\left\{\frac{1}{2},\frac{1}{3},\frac{2}{3},\frac{1}{4},\frac{2}{4},\frac{3}{4},\ldots\right\}.
$$
(Note that we exclude 0 and 1 from each block, this is purely a matter of taste.)
The above strategy is successful: $\omega$ is equidistributed in $[0,1]$. Actually, it is a nice exercise to prove that (a) $D_N^*(\omega)=\mathcal{O}(1/\sqrt{N})$ (hence $\omega$ is equidistributed by (\ref{equidistrChar})), and (b) the convergence rate $N^{-1/2}$ is sharp in the following sense: there exists an absolute constant $c>0$ such that $D_N^*(\omega)>c/\sqrt{N}$ for infinitely many natural numbers $N$. (For a refresher on the Bachmann-Landau 'big O' notation, see the end of this section.) 


When analyzing the sequence $\omega$, we see that it has some 'beauty spots' which might impede the convergence rate of $D_N^*(\omega)$. The most obvious issue is that the ordering imposed within any given block does not effectively 'spread out' the elements of the block in $[0,1]$. For instance, assume that we consider the block $\{1/100,2/100,...,99/100\}$. If we take any $k~~(1\le k\le 50)$ and consider the first $k$ elements of the block, then all of these elements fall in $[0,1/2]$. See also Remark \ref{orderingRemark} below for a more precise discussion of this issue.

The previous discussion leads us to the following construction.
Let $p_m$ denote the $m$-th prime number. For a fixed prime $p$, denote $\overline{j}=j^{-1}\pmod{p}$ for $j\in\mathbb{Z}_p^*=\{1,2,\ldots,p-1\}$. Define the sequence $\eta$ as follows:
\begin{equation}
    \label{etaDefn}
    \eta=\left\{\frac{\overline{1}}{2},\frac{\overline{1}}{3},\frac{\overline{2}}{3},\ldots,\frac{\overline{1}}{p_m},\frac{\overline{2}}{p_m},\frac{\overline{3}}{p_m},\ldots,\frac{\overline{p_m-1}}{p_m},\ldots\right\}.
\end{equation} 
That is, the sequence $\eta$ defined by (\ref{etaDefn}) consists of blocks in which all the terms have the same denominator $p$. The point is, roughly speaking,
that the numerators $k^{-1}\pmod{p}~~(1\le k\le p-1)$ are pseudorandomly distributed on the residues $\mathbb{Z}_p^*$ (see Lemma \ref{inverseLemma} and Remark \ref{orderingRemark} below).

We now wish to solve the following problem: determine the sharp convergence rate of the star discrepancy (\ref{discrepancyDefn}) of the sequence (\ref{etaDefn}). Our main result is Theorem \ref{mainteo} below.
\begin{theorem}
\label{mainteo}
Let $\eta$ be defined as above. Then 
\begin{equation}
    \label{mainestimate}
    D_N^*(\eta)=\mathcal{O}\left(\frac{1}{\sqrt{N\ln(N)}}\right).
\end{equation}
The convergence rate given by (\ref{mainestimate}) is sharp, i.e.,
\begin{equation}
    \label{mainestimate2}
    \liminf_{N\rightarrow\infty}\sqrt{N\ln(N)}D_N^*(\eta)>\frac{1}{2}.
\end{equation}
\end{theorem}
The statements of the previous theorem are probably not very important in their own right. More interesting is, in our opinion, the rather novel combination of previously known results that allows us to get quite precise estimates for $D_N^*(\eta)$.

We also believe that Lemma \ref{inverseLemma} (upon which the proof of Theorem \ref{mainteo} is based) provides a new point of view on the pseudorandom nature of the involution $x\mapsto x^{-1}$ on $\mathbb{Z}_p^*$ for large primes $p$. But we must stress that Lemma \ref{inverseLemma} is not new, it is essentially just a reformulation of a well-known discrepancy estimates for explicit inversive pseudorandom numbers (see \cite[Theorem 5.3.25]{NW}). 


As usual, we write $f(N)=\mathcal{O}(g(N))$ if there exist absolute constants $C>0$ and $N_0\in\mathbb{N}$ such that $f(N)\le Cg(N)$ for all $N\in\mathbb{N}, N\ge N_0$. We also write $f(N)=o(g(N))$ if $f(N)/g(N)\rightarrow0$ as $N\rightarrow\infty$.

\section{Auxiliary result on discrepancy}
Above we defined $D_N^*(\xi)$ for an infinite sequence $\xi$. One can define the star discrepancy for finite sequences in the obvious way: let $\xi=\{\xi_n\}_{n=1}^N$, then $D_k^*(\xi)$ is defined as in (\ref{discrepancyDefn}) for $k=1,2,...,N$. 
\begin{rem}
\label{setDisc}
It should be stressed that $\xi=\{\xi_n\}_{n=1}^N$ is still a \emph{sequence} and not a \emph{set}, i.e., the order of the elements matter when computing $D_k^*(\xi)$ for $k<N$. Only in the special case $k=N$ is the order irrelevant when calculating $D_N^*(\xi)$. 
\end{rem}
The next result is useful when estimating discrepancy. It is often referred to as the triangle inequality for discrepancy. 
\begin{prop}
\label{triangleIneq}
Let $\xi_j~~(1\le j\le k)$ be finite sequences with $\sharp\xi_j=N_j$ for $1\le j\le k$. Let $\xi$ be the concatenation of $\xi_1,\xi_2,...,\xi_k$, i.e., $\xi$ is the finite sequence with $N=N_1+N_2+\ldots+ N_k$ terms formed by first taking the elements of $\xi_1$, then the elements of $\xi_2$, etc. Then
$$
ND_N^*(\xi)\le\sum_{j=1}^kN_jD_{N_j}^*(\xi_i).
$$
\end{prop}
For a proof, see \cite[Chapter 2.2]{KN}.
We also record the following simple observation. If $\xi=\{j/N\}_{j=1}^{N-1}$, then $\sharp\xi=N-1$ and
\begin{equation}
    \label{minimalDisc}
    D^*_{N-1}(\xi)=\frac{1}{N-1}.
\end{equation}

The next lemma is the main tool for the proof of Theorem \ref{mainteo}.
\begin{lem}
\label{inverseLemma}
Fix $m\in\mathbb{N}$, let $p_m$ be the $m$-th prime number, and set $$
E_{p_m}=
\left\{\frac{\overline{j}}{p_m}\right\}_{j=1}^{p_m-1},
$$
where $\overline{j}=j^{-1}\pmod{p_m}$ for $j\in\mathbb{Z}_{p_m}^*=\{1,2,\ldots,p_m-1\}$.
Then
\begin{equation}
\label{restEstimate2}
\max_{1\le k\le p_m-1}kD_k^*(E_{p_m})=o(m),
\end{equation}
as $m\rightarrow\infty$.
\end{lem}
\begin{rem}
Note that the finite sequence $E_{p_m}$ is the $m$-th block of the sequence $\eta$ (\ref{etaDefn}).
\end{rem}
\begin{proof}[Proof of Lemma \ref{inverseLemma}]
The proof is essentially based on Theorem 5.3.25 from \cite{NW}. In our notation, that theorem states
\begin{equation}
\nonumber
kD_k^*(E_p)\le (2\sqrt{p}+1)\left(\ln(p)+\frac{1}{3}\right)^2+\frac{k}{p}   
\end{equation}
for any $1\le k\le p-1$. The dominating term is clearly $\sqrt{p}\ln(p)^2$, so we get
\begin{equation}
\label{restEstimate}
kD_k^*(E_p)\le C\sqrt{p}\ln(p)^2
\end{equation}
for some absolute constant $C$.
Then (\ref{restEstimate}) and the prime number theorem (see (\ref{primeNoThm}) below) imply (\ref{restEstimate2}). Indeed, taking $p=p_m$ in (\ref{restEstimate}) and using (\ref{primeNoThm}), we get that for any $k\in\{1,2,\ldots, p_m-1\}$,
$$
kD_k^*(E_{p_m})\le C\sqrt{p_m}\ln(p_m)^2=\mathcal{O}(m^{1/2}(\ln(m))^{5/2})=o(m).
$$
\end{proof}
\begin{rem}
\label{orderingRemark}
The special 'inverse ordering' of the elements of the sequences $E_{p_j}$ is crucial for (\ref{restEstimate2}) to hold. Indeed, if $p$ is an arbitrary prime number and $\zeta_p=\{j/p\}_{j=1}^{p-1}$
(that is, the elements are ordered increasingly), then
\begin{equation}
    \label{increasingOrder}
    \max_{1\le k\le p-1}kD_k^*(\zeta_p)\ge\frac{p-1}{8}.
\end{equation}
Take $k=(p-1)/2$, then $A_k(\zeta_p,r)=k$ for all $r>1/2$. Thus, by taking $r=3/4$, we get
$$
kD_k^*(\zeta_p)\ge k\left|\frac{A_k(\zeta_p,3/4)}{k}-\frac{3}{4}\right|=\frac{k}{4}=\frac{p-1}{8},
$$
thus proving (\ref{increasingOrder}). In particular, taking $p=p_m$ in (\ref{increasingOrder}) and using (\ref{primeNoThm}), we get
$$
\frac{1}{m}\max_{1\le k\le p_m-1}kD_k^*(\zeta_p)\ge\frac{\ln(m)}{8}(1+o(1))\rightarrow\infty
$$
as $m\rightarrow\infty$.
\end{rem}

\section{Some asymptotic formulae}

The goal of this section is to prove Lemma \ref{mainLemma} below. For the sake of the reader, we first recall a number of well-known results, starting with the famous \emph{prime number theorem}:
\begin{equation}
    \label{primeNoThm}
    p_m=m\ln(m)(1+o(1)).
\end{equation}
We shall also need the asymptotic behaviour of the following sum:
$$
P(m)=\sum_{k=1}^m(p_k-1).
$$
It is well-known that
\begin{equation}
    \label{sumPrimes}
    P(m)=\frac{m^2}{2}\ln(m)(1+o(1))
\end{equation}
(see, e.g., \cite{A2019} and the references given there).

\begin{rem}
The definition of $P(m)$ might seem strange. The point is that $p_k-1=\sharp E_{p_k}$, where $E_{p_k}$ is the $k$-th block of the sequence $\eta$ defined by (\ref{etaDefn}).
\end{rem}

Consider the function $f(x)=xe^x$ for $x\ge-1$. The inverse of $f$ exists and is denoted $W$. The domain of $W$ is $\left[-1/e,\infty\right)$. In fact, $W(x)$ for $x\ge-1/e$ is the principal branch of a multifunction called the \emph{Lambert $W$ function}. For a thorough discussion of $W$ and its applications, see \cite{lambertW}. We simply record the obvious identity
\begin{equation}
    \label{lambertIdentity}
    e^{W(x)}=\frac{x}{W(x)}\quad(x\ge-1/e),
\end{equation}
and the asymptotic formula
\begin{equation}
    \label{lambertAsymp}
    W(x)=\ln(x)(1+o(1))\quad(x\rightarrow\infty).
\end{equation}

\begin{lem}
\label{mainLemma}
For any $N\in\mathbb{N}$, let $m=m(N)$ be the unique natural number such that 
\begin{equation}
    \label{mDefinition}
    P(m)\le N<P(m+1). 
\end{equation}
Then the following asymptotic relations between $N$ and $m$ hold:
\begin{equation}
    \label{mAsymptotic2}
    N=\left(\frac{1}{2}+o(1)\right)m^2\ln(m),
\end{equation}
\begin{equation}
    \label{mAsymptotic}
    m=(2+o(1))\sqrt{\frac{N}{\ln(N)}}.
\end{equation}
(The term $o(1)$ here refers to either $N\rightarrow\infty$ or $m\rightarrow\infty$; of course both take place simultaneous.)
\end{lem}
\begin{proof}
Note that (\ref{mAsymptotic2}) follows immediately from (\ref{sumPrimes}). We proceed with (\ref{mAsymptotic}).
By (\ref{mAsymptotic2}), there exists $c(N)$ such that for any $N$ there holds 
$$
N=\displaystyle\frac{c(N)m^2}{2}\ln(m),
$$
and $c(N)=1+o(1)$ as $N\rightarrow\infty$.
Let $t=\ln(m)$ and note that 
$$
\frac{4N}{c(N)}=2m^2\ln(m)=2te^{2t}.
$$
Hence,
$$
2\ln(m)=2t=W\left(\frac{4N}{c(N)}\right).
$$
Using (\ref{lambertIdentity}), it follows that
$$
m^2=e^{2\ln(m)}=e^{2t}=e^{W(4N/c(N))}=\frac{4N}{c(N)W(4N/c(N))}.
$$
Since $W$ is increasing and $3N<4N/c(N)<5N$ for $N$ sufficiently large, we get
$$
\frac{2}{\sqrt{c(N)}}\sqrt{\frac{N}{W(5N)}}\le m\le \frac{2}{\sqrt{c(N)}}\sqrt{\frac{N}{W(3N)}}.
$$
By (\ref{lambertAsymp}), $W(kN)=\ln(N)(1+o(1))$ for any $k\in\mathbb{N}$ as $N\rightarrow\infty$. It follows that
$$
m=(2+o(1))\sqrt{\frac{N}{\ln(N)}}
$$
as $N\rightarrow\infty$.
\end{proof}

\section{Proof of Theorem \ref{mainteo}}
In this section, we give the proof of the main theorem of this note.
\begin{proof}[Proof of Theorem \ref{mainteo}]
Let $N\in\mathbb{N}$ be fixed but arbitrary. Let $m$ be given by (\ref{mDefinition}). Then we have 
$$
\{\eta_1,\eta_2,\ldots,\eta_N\}=\left(\bigcup_{j=1}^m E_{p_j}\right)\cup E_m^*,
$$
where $E_m^*=\{\eta_{P(m)+1},\eta_{P(m)+2},\ldots, \eta_N\}$. Let $k=\sharp E_m^*$. By definition, $k<p_{m+1}-1$ and the finite sequence $E_m^*$ consists of the first $k$ terms of $E_{p_{m+1}}$. By Proposition \ref{triangleIneq}, we have
$$
ND_N^*(\eta)\le\sum_{j=1}^m(p_j-1)D_{p_j-1}^*(E_{p_j})+kD_k^*(E_m^*).
$$
By (\ref{minimalDisc}), we have $(p_j-1)D_{p_j-1}^*(E_{p_j})=1$. Further, by (\ref{restEstimate2}),
$$
kD_k^*(E_m^*)=kD_k^*(E_{p_{m+1}})=o(m).
$$
Thus,
$$
ND_N^*(\eta)\le m(1+o(1)).
$$
It follows from (\ref{mAsymptotic}) that
$$
D_N^*(\eta)\le\frac{m}{N}(1+o(1))=\frac{2+o(1)}{\sqrt{N\ln(N)}},
$$
and this concludes the proof of (\ref{mainestimate}). To prove (\ref{mainestimate2}), let $N=P(m)$ for $m\in\mathbb{N}$. Take $r=1-1/(2p_m)$, then $A_N(\eta,r)=N$ and
\begin{align}
\nonumber
D_N^*(\eta)\ge& \left|\frac{A_N(\eta,r)}{N}-r\right|=\left|1-\left(1-\frac{1}{2p_m}\right)\right|\\
\nonumber
&=\frac{1}{2p_m}=\frac{1}{m\ln(m)\bigl(2+o(1)\bigr)}
\end{align}
by (\ref{primeNoThm}).
On the other hand, by (\ref{mAsymptotic2}), we have
\begin{align}
\nonumber
\frac{1}{m\ln(m)(2+o(1))}&=\frac{m}{N(4+o(1))},\\
\nonumber
&=\frac{m}{N}\frac{1}{4+o(1)}=\frac{2+o(1)}{4+o(1)}\frac{1}{\sqrt{N\ln(N)}}
\end{align}
where we used (\ref{mAsymptotic}) for the final relation.
Hence, 
$$
\liminf_{N\rightarrow\infty}\sqrt{N\ln(N)}D_N^*(\eta)\ge\frac{1}{2}.
$$
\end{proof}

\end{document}